\documentclass[11pt]{article}

\usepackage{amsmath,amsthm,amssymb,amsfonts}
\usepackage{geometry}
\usepackage{hyperref}
\usepackage{enumitem}
\newcommand{\F}{\mathbb{F}}
\newcommand{\Z}{\mathbb{Z}}
\newcommand{\PG}{\mathrm{PG}}
\newcommand{\GL}{\mathrm{GL}}
\newcommand{\PGL}{\mathrm{PGL}}
\newcommand{\ord}{\mathrm{ord}}
\newcommand{\CR}{\mathrm{CR}}
\newcommand{\Fr}{\mathrm{Fr}}
\newcommand{\diag}{\mathrm{diag}}

\newtheorem{theorem}{Theorem}[section]
\newtheorem{lemma}[theorem]{Lemma}
\newtheorem{proposition}[theorem]{Proposition}
\newtheorem{corollary}[theorem]{Corollary}
\theoremstyle{definition}
\newtheorem{definition}[theorem]{Definition}
\theoremstyle{remark}
\newtheorem{remark}[theorem]{Remark}
\newtheorem{example}[theorem]{Example}

\title{Regular Cyclic $(q+1)$-Arcs in $\PG(3,2^m)$:
Spectral Rigidity, Descent, and an MDS Criterion}

\author{Bocong Chen$^1$, Jing Huang$^2$ and
Hao Wu$^1$\footnote{E-mail addresses: {\it
mabcchen@scut.edu.cn (B. Chen),
jhuangmath@foxmail.com (J. Huang),
mawuhao\_math@mail.scut.edu.cn (H. Wu).
}}
}

\date{\small
$1.$ School of Mathematics, South China University of Technology, Guangzhou 510641, China\\
$2.$
School of Mathematics and Information Science,
Guangzhou University, Guangzhou 510006, China\\
}

\begin{document}
\maketitle

\begin{abstract}
Let $q=2^m$ with $m\ge 3$ and set $n:=q+1$.
We investigate $(q+1)$-arcs $\mathcal A\subset \mathrm{PG}(3,q)$ that admit a regular cyclic subgroup
$C\le \mathrm{PGL}(4,q)$ of order $n$.
Over  $K=\mathbb{F}_{q^2}$, such an action
can be conjugated to a diagonal one,
producing explicit cyclic monomial  models
\[
\mathcal M_a = \{[1:t:t^a:t^{a+1}]:t\in U_n\}\subset \mathrm{PG}(3,K),
\qquad U_n=\{u\in K^\times:u^n=1\},
\]
with $a\in(\mathbb{Z}/n\mathbb{Z})^\times$.
We develop
 a  spectral rigidity principle  to
  obtain a precise descent criterion:
$\mathcal M_a$ is $K$-projectively equivalent to a $(q+1)$-arc defined over $\mathbb{F}_q$
if and only if $a\equiv \pm 2^e \pmod n$ for some integer $e$ with $\gcd(e,m)=1$.
Consequently, regular cyclic pairs $(\mathcal A,C)$ fall into exactly $\varphi(m)/2$
$K$-projective equivalence classes.
As an immediate coding-theoretic application, we resolve the remaining AMDS/MDS dichotomy for the BCH family
$\mathcal C_{(q,q+1,3,h)}$ studied by Xu et al.:
$\mathcal C_{(q,q+1,3,h)}$ is MDS
 if and only if
$2h+1\equiv \pm 2^e \pmod n$ for some $e$ with $\gcd(e,m)=1$.
The underlying spectral rigidity step is formulated in a general setting for diagonal regular cyclic pairs
in $\mathrm{PG}(r,K)$, providing a portable reduction of projective equivalence questions to explicit congruences on exponent data.

\medskip
\textbf{2020 MSC:} 51E21, 51E20, 20B25, 11T71, 94B15.

\textbf{Keywords:} $(q+1)$-arc;  projective equivalence; MDS code; BCH code; finite projective space.
\end{abstract}

% ============================================================
\section{Introduction}
\label{sec:intro}
Point sets with strong incidence constraints in
finite projective spaces lie at the interface of finite geometry,
algebraic curves, and coding theory. 
A set $\mathcal{A}\subset\mathrm{PG}(3,q)$ of size $n$ is an $n$-arc
if no four points of $\mathcal{A}$ lie in a plane.
These sets are the geometric counterparts of linear $[n,n-4,5]$ MDS codes;
see, for example, \cite{Bruen,HirschfeldBook,Huffman} and the survey \cite{BallLavrauw}.
The extremal case $n=q+1$ is particularly significant as it represents the maximum possible size of an arc in $\mathrm{PG}(3,q)$.
In the even characteristic setting $q=2^{m}$ and $m\geq 3$, 
$(q+1)$-arcs in $\mathrm{PG}(3,q)$ exhibit remarkable structural rigidity. 
Indeed,  Casse and Glynn proved that every $(q+1)$-arc in $\mathrm{PG}(3,q)$ is projectively equivalent over $\mathbb{F}_{q}$
to a Frobenius-twisted normal rational curve
\[
A_{e}=\bigl\{[1:t:t^{2^{e}}:t^{2^{e}+1}]:t\in\mathbb{F}_{q}\bigr\}\cup\{[0:0:0:1]\},
\qquad\gcd(e,m)=1,
\]
generalizing the twisted cubic.
The Casse--Glynn theorem also states that the full stabilizer of 
$A_e$---that is, the group of all projectivities preserving the 
set---is isomorphic to $\mathrm{PGL}(2,q)$; see \cite{CasseGlynn,HirschfeldBook}.

\medskip
A recurring theme in both geometry and coding theory 
is that large symmetry forces rigidity.
On the coding side, cyclic and constacyclic structures lead to efficient constructions and algebraic tests,
while on the geometric side they correspond to sharply transitive group actions on point sets.
Motivated by this, we study $(q+1)$-arcs equipped with a distinguished cyclic symmetry.
Let $n:=q+1$.
We call $(\mathcal{A},C)$ a \emph{regular cyclic pair} if $\mathcal{A}\subset\mathrm{PG}(3,q)$ is a $(q+1)$-arc and
$C\leq\mathrm{PGL}(4,q)$ is a cyclic subgroup of order $n$ acting regularly (sharply transitively) on $\mathcal{A}$.
This pair-theoretic viewpoint parallels classical and modern investigations of cyclic or pseudo-cyclic MDS behaviour at length $q+1$
\cite{BallGR,DahlPedersen,Maruta,PedersenDahl,WangDing}, 
but our perspective is that tracking the group-theoretic rigidity offers a clearer insight into the MDS property than focusing solely on geometric constraints.
Accordingly, we keep track of the chosen cyclic subgroup and ask how it refines projective classification and descent.

\medskip

While the Casse--Glynn classification settles the structure of arcs 
as  point sets, it leaves open the finer question central here:
\emph{given a regular cyclic subgroup, how does it sit inside the stabilizer, and what arithmetic constraints does it impose on
projective equivalence?}
Stabilizer phenomena for these arcs have been investigated from several perspectives  (e.g.\ \cite{CP,Hir1}),
but an explicit arithmetic rigidity mechanism for regular cyclic pairs  has not been isolated.

\medskip
We specialize throughout to the even case $q=2^m$ with $m\geq3$.
A key observation is that after extending scalars to $K=\mathbb{F}_{q^{2}}$, a Singer element of order $n$ becomes diagonalizable.
This leads to a family of diagonal cyclic orbit models in projective space: for $a\in(\mathbb{Z}/n\mathbb{Z})^{\times}$ we consider
\[
\mathcal{M}_{a}=\{[1:t:t^{a}:t^{a+1}]:t\in U_{n}\}\subset\mathrm{PG}(3,K),
\qquad U_{n}=\{u\in K^{\times}:u^{n}=1\},
\]
together with the diagonal element
\[
[S_{a}]=[\mathrm{diag}(1,\beta,\beta^{a},\beta^{a+1})]\in\mathrm{PGL}(4,K)
\]
of order $n$ acting regularly on $\mathcal{M}_{a}$.
In this diagonal setting, projective equivalence becomes arithmetic: conjugacy in $\mathrm{PGL}$ preserves the eigenvalue multiset,
and for diagonal generators this eigenvalue data is precisely the underlying exponent multiset modulo $n$ (up to an overall shift and
replacing the generator by a power).
We package this ``spectral step'' into an abstract rigidity statement (Theorem~\ref{thm:spectral-rigidity-practical}),
which will be invoked later to turn $K$-projective equivalence of cyclic models into explicit congruences on exponent data.

\medskip
Our main geometric result is an exact descent criterion for the models $\mathcal{M}_{a}$, answering the question:
for which $a$ is $\mathcal{M}_{a}$ $K$-projectively equivalent to a $(q+1)$-arc contained in the subgeometry $\mathrm{PG}(3,q)$?

\smallskip\noindent\textbf{Theorem}
(see Theorem~\ref{thm:geom-main} in Section \ref{sec:cyclic-monomial})
Let $q=2^{m}$ with $m\geq 3$ and $n=q+1$.
A cyclic monomial model $\mathcal{M}_{a}\subset\mathrm{PG}(3,\mathbb{F}_{q^{2}})$ is $K$-projectively equivalent to a $(q+1)$-arc
contained in $\mathrm{PG}(3,q)$ if and only if $a\equiv\pm 2^{e}\pmod{n}$ for some $e$ with $\gcd(e,m)=1$.

Beyond point sets, we refine the classification by keeping track of the chosen cyclic subgroup:
up to $K$-projective equivalence there are exactly $\varphi(m)/2$ classes of regular cyclic pairs $(\mathcal{A},C)$
(Corollary~\ref{cor:regular-cyclic-classification}).

\medskip
As a concrete application, we revisit the BCH family $\mathcal{C}_{(q,q+1,3,h)}$ studied by Xu et al.~\cite{Xu}.
For even $q=2^{m}$, under standard non-degeneracy hypotheses, they showed that the family exhibits a genuine dichotomy:
the code is either AMDS ($d=4$) or MDS ($d=5$), and they left open the problem of deciding which occurs in general.
We resolve this by translating the parity-check geometry of $\mathcal{C}_{(q,q+1,3,h)}$ into the monomial orbit $\mathcal{M}_{2h+1}$
over $K$ and invoking the descent criterion above, yielding an exact congruence test.

\smallskip\noindent\textbf{Theorem}
(see Theorem~\ref{thm:main-coding} in Section \ref{MDS})
Let $q=2^{m}$ with $m\geq 3$ and let $n=q+1$.
Under the standard non-degeneracy hypotheses, the BCH code $\mathcal{C}_{(q,q+1,3,h)}$ is MDS  
if and only if there exists an integer $e$ with $\gcd(e,m)=1$ such that $2h+1\equiv\pm 2^{e}\pmod{n}$.

\medskip
The proofs of the two theorems above follow a common template:
diagonalize the cyclic action over $K$, extract arithmetic constraints from eigenvalue data,
and then analyze the resulting congruences.
The abstract formulation of this spectral rigidity mechanism is given in Section~\ref{sec:portable-spectral-rigidity}.

\medskip\noindent\textbf{Organization of the paper.}
Section~\ref{sec:prelim} fixes notation, recalls the Casse--Glynn normal forms for $(q+1)$-arcs in $\mathrm{PG}(3,2^{m})$, and
records the lifted $\mathrm{PGL}(2,q)$-action and the diagonalization of Singer elements over $K=\mathbb{F}_{q^{2}}$.
Section~\ref{sec:cyclic-monomial} introduces the cyclic monomial models $\mathcal{M}_{a}$ and proves the descent criterion of
Theorem~\ref{thm:geom-main}.
Section~\ref{MDS} applies this classification to the BCH family $\mathcal{C}_{(q,q+1,3,h)}$ and derives the exact MDS congruence criterion.
Section~\ref{sec:portable-spectral-rigidity} formulates the spectral rigidity principle in a general form and briefly discusses its broader perspective.
Finally, the appendix contains deferred proofs of Lemmas~\ref{lem:PGL2-lift} and~\ref{lem:phi-conjugates-singer}.

% ============================================================

% ============================================================
\section{Preliminaries}\label{sec:prelim}

Throughout the paper, $q$ denotes a prime power
and $\F_q$ the finite field of order $q$.
We write $\PG(3,q)$ for the $3$-dimensional
projective space over $\F_q$, namely the set of $1$-dimensional
$\F_q$-subspaces of $\F_q^4$.  Points are written
in homogeneous coordinates
$[x_0:x_1:x_2:x_3]$ with $(x_0,x_1,x_2,x_3)\neq 0$, where
$[x_0:x_1:x_2:x_3]=
[\lambda x_0:\lambda x_1:\lambda x_2:\lambda x_3]$
for all $\lambda\in\F_q^\times$.
A (projective) hyperplane in $\PG(3,q)$ is the set of points satisfying a nontrivial homogeneous linear equation
$a_0x_0+a_1x_1+a_2x_2+a_3x_3=0$ with $(a_0,a_1,a_2,a_3)\neq 0$; in $\PG(3,q)$ we also call such a hyperplane a \emph{plane}.

A \emph{projectivity} of $\PG(3,q)$ is the map induced by an invertible linear transformation of $\F_q^4$,
i.e.\ an element of $\PGL(4,q)=\GL(4,q)/\F_q^\times$.
Two point sets $\mathcal{A},\mathcal{B}\subseteq \PG(3,q)$ are \emph{projectively equivalent over $\F_q$}
if there exists $P\in\PGL(4,q)$ such that $P(\mathcal{A})=\mathcal{B}$.
More generally, if $K/\F_q$ is a field extension, we say that $\mathcal{A}$ and $\mathcal{B}$ are
\emph{projectively equivalent over $K$} if there exists $P\in\PGL(4,K)$ sending $\mathcal{A}$ to $\mathcal{B}$.
A set $\mathcal{A}\subset \PG(3,q)$ is an \emph{$n$-arc}
if $|\mathcal{A}|=n$ and no four points of $\mathcal{A}$
lie in a common plane.  Equivalently, for any choice of nonzero column representatives
$v_P\in\F_q^4$ of $P\in\mathcal{A}$, every $4\times 4$
matrix whose columns are $v_{P_1},v_{P_2},v_{P_3},v_{P_4}$
(with $P_i$ distinct) is invertible.
A \emph{$(q+1)$-arc} is an $n$-arc with $n=q+1$.

A $4\times(q+1)$ matrix $H$ over $\F_q$ defines a point set $\{[H_1],\dots,[H_{q+1}]\}\subset\PG(3,q)$.
The associated $[q+1,q-3]$ code with parity--check matrix $H$
is MDS (equivalently $d=5$) if and only if these
points form a $(q+1)$-arc, i.e.\ every $4$ columns of $H$
are linearly independent; see \cite[\S2.4]{Huffman} or \cite[Ch.~21]{HirschfeldBook}.

From now on we specialize to the even case $q=2^m$
with $m\ge 3$ and set $n:=q+1$.
Let $K:=\F_{q^2}$ and let the set of
$(q+1)$th roots of unity in
$K$ be denoted by
\[
U_n:=U_{q+1}:=\{u\in K^\times : u^{q+1}=1\}.
\]
Then $U_{q+1}$ is cyclic of order $q+1$.
We will use cross-ratios for points on $U_{q+1}$.
For $x,y,z,w$ in a field $K$ with $x\neq w$ and $y\neq z$, the (affine) cross-ratio is
\[
\CR(x,y;z,w):=\frac{(x-z)(y-w)}{(x-w)(y-z)}.
\]
In what follows we only apply $\CR$ to pairwise distinct points.
The next lemma is a standard and convenient observation:
cross-ratios of four distinct $(q+1)$th roots of unity
actually lie in the subfield $\F_q$.

\begin{lemma}\label{lem:CR-in-Fq}
Let $q=2^m$ with
$m\geq 3$ and $n=q+1$.
If $x,y,z,w\in U_{q+1}$ are pairwise distinct and $t=\CR(x,y;z,w)$,
then $t\in \F_q\setminus\{0,1\}$.
\end{lemma}
\noindent\emph{Proof.}
Using $u^q=u^{-1}$ for $u\in U_{q+1}$ we
obtain $t^q=t$, hence $t\in\F_q$; moreover $t\neq 0,1$
for distinct $x,y,z,w$. \qed

Finally, we recall the standard family of $(q+1)$-arcs in $\PG(3,q)$ that exhausts all such arcs up to projective equivalence
when $q$ is even. Fix an integer $e$ with $1\le e\le m-1$ and $\gcd(e,m)=1$, and put $\sigma:=2^e$.
Define the point set
\begin{equation}\label{eq:Ae-def}
A_{e}
:=\Bigl\{[\,1:t:t^{\sigma}:t^{\sigma+1}\,]\ :\ t\in\F_q\Bigr\}\ \cup\ \{[0:0:0:1]\}
\ \subset\ \PG(3,q).
\end{equation}
It is well known that $A_e$ is a $(q+1)$-arc; see Casse--Glynn~\cite{CasseGlynn} or \cite[Chapter~21]{HirschfeldBook}.
The following classification theorem underlies our approach.

\begin{theorem}[Casse--Glynn]
\label{thm:CasseGlynn}
Let $q=2^m$ with $m\ge 3$.
Every $(q+1)$-arc in $\PG(3,q)$ is projectively equivalent over $\F_q$ to $A_e$ for some $e$ with $\gcd(e,m)=1$.
\end{theorem}

\begin{remark}
We will  use that the  stabilizer of $A_e$ in $\PGL(4,q)$
is isomorphic to $\PGL(2,q)$ which
acts on the parameter line in the standard way;
this lifted action is described in
Lemma \ref{lem:PGL2-lift} below.
\end{remark}

We call $M\in\PGL(2,q)$ a \emph{Singer element} if $\ord(M)=q+1$.
We will use two standard auxiliary facts: a lift of
the natural $\PGL(2,q)$-action to $\PGL(4,q)$
stabilizing $A_e$, and a diagonalization of
Singer elements over $K=\F_{q^2}$.
Proofs are deferred to Appendix~\ref{app:deferred-proofs}.

\begin{lemma}
\label{lem:PGL2-lift}
Let $q=2^m$ with $m\ge 3$, let $e$ be an integer with $\gcd(e,m)=1$, and set
$\sigma:=2^e$. Consider the $(q+1)$-arc
\[
A_{e}
=\Bigl\{P_t=[\,1:t:t^{\sigma}:t^{\sigma+1}\,]\ :\ t\in\F_q\Bigr\}\ \cup\ \{P_\infty=[0:0:0:1]\}
\ \subset\ \PG(3,q)
\]
from \eqref{eq:Ae-def}. Let
$
G_e:=\mathrm{Stab}_{\PGL(4,q)}(A_{e})
$
be its stabilizer in $\PGL(4,q)$.
\begin{enumerate}
\item Define a map $\psi:\PG(1,q)\to A_{e}$ by
\begin{equation}\label{eq:psi-def}
\psi([x:y]) := [\,x^{\sigma+1} : x^{\sigma}y : xy^{\sigma} : y^{\sigma+1}\,].
\end{equation}
Then $\psi$ is a bijection. In particular, for $t\in\F_q$,
$
\psi([1:t])=P_t
~\text{and}~
\psi([0:1])=P_\infty.
$
\item For $M=\begin{bmatrix}a&b\\ c&d\end{bmatrix}\in\mathrm{GL}(2,q)$, define
\begin{equation}\label{eq:Mabcd-def}
M_{a,b,c,d}:=
\begin{pmatrix}
a^{\sigma+1} & a^{\sigma}b & ab^{\sigma} & b^{\sigma+1}\\
a^{\sigma}c & a^{\sigma}d & b^{\sigma}c & b^{\sigma}d\\
ac^{\sigma} & bc^{\sigma} & ad^{\sigma} & bd^{\sigma}\\
c^{\sigma+1} & c^{\sigma}d & cd^{\sigma} & d^{\sigma+1}
\end{pmatrix}\in\mathrm{GL}(4,q).
\end{equation}
Let $\varphi:\PGL(2,q)\to\PGL(4,q)$ be the map sending the class of $M$
to the class of $M_{a,b,c,d}$. Then for all $\tau\in\PG(1,q)$ we have the
equivariance identity
\begin{equation}\label{eq:equivariance}
\varphi(M)\cdot \psi(\tau) \;=\; \psi(M\cdot\tau),
\end{equation}
where $M\cdot[x:y]=[ax+by:cx+dy]$ is the natural action of $\PGL(2,q)$
on $\PG(1,q)$. In particular, $\varphi(\PGL(2,q))\le G_e$.

\item The map $\varphi$ is a group monomorphism.
Moreover, since $q\geq8$,
we have $\varphi(\PGL(2,q))=G_e$ and hence $G_e\simeq\PGL(2,q)$.
\end{enumerate}
\end{lemma}
\noindent\emph{Proof.} See Appendix~\ref{app:deferred-proofs}. \qed

In particular, $\psi$ identifies $A_e$ with $\PG(1,q)$ and $\varphi$ embeds the natural $\PGL(2,q)$-action on $\PG(1,q)$
into $\PGL(4,q)$ as a subgroup stabilizing $A_e$;
for $q\geq8$ this gives the full stabilizer $G_e$.
To pass from this lifted action to diagonal cyclic models over $K=\F_{q^2}$, we next conjugate a Singer element
 to a multiplication map on $U_{q+1}$.

\begin{lemma}
\label{lem:phi-conjugates-singer}
Let $q=2^m$ with $m\ge 3$, and let $K=\F_{q^2}$.  Let
$
U_{q+1}=\{u\in K^\times : u^{q+1}=1\}
$
be the set of $(q+1)$th roots of unity in $K$.
Let $M\in\PGL(2,q)$ be a  Singer element.
Then the following hold.

\begin{enumerate}
\item[\textup{(a)}] $M$ has no fixed point in $\PG(1,q)$, but it has exactly
two fixed points in $\PG(1,K)$; moreover these two fixed points are
Galois conjugates under Frobenius $x\mapsto x^q$.

\item[\textup{(b)}] Let $\gamma\in K\setminus \F_q$ be one fixed point of $M$
in $\PG(1,K)$, so the other fixed point is $\gamma^q$.
Define the M\"{o}bius
transformation
\[
\phi:\PG(1,K)\longrightarrow \PG(1,K),\qquad
\phi(t)=\frac{t-\gamma}{t-\gamma^q}\ \ (t\in K),\qquad \phi(\infty)=1.
\]
Then $\phi$ restricts to a bijection
\[
\phi:\PG(1,q)\ \overset{\sim}{\longrightarrow}\ U_{q+1}.
\]

\item[\textup{(c)}] There exists an element $\beta\in U_{q+1}$ of order $q+1$
such that for all $\tau\in\PG(1,q)$,
\[
\phi\bigl(M\cdot \tau\bigr)=\beta\,\phi(\tau).
\]
Equivalently, as maps on $\PG(1,q)$,
\[
\phi\circ M\circ \phi^{-1}:\ U_{q+1}\longrightarrow U_{q+1},\qquad u\longmapsto \beta u.
\]
Moreover, if $M$ is represented by a matrix
$\begin{bmatrix}a&b\\ c&d\end{bmatrix}\in\mathrm{GL}(2,q)$, then $c\neq 0$ and one may take
\[
\beta=\phi\!\left(\frac{a}{c}\right)=\frac{a-c\gamma}{a-c\gamma^q}\in U_{q+1},
\]
which indeed has order $q+1$.
\end{enumerate}
\end{lemma}
\noindent\emph{Proof.} See Appendix~\ref{app:deferred-proofs}. \qed

We will use this multiplicative model in the next section to construct explicit diagonal-monomial orbit models in $\PG(3,K)$.

% ============================================================

% ============================================================
\section{Regular cyclic actions and cyclic monomial models}
\label{sec:cyclic-monomial}
By Lemma~\ref{lem:phi-conjugates-singer}(c), after identifying $\PG(1,q)$ with $U_{q+1}$ over $K=\F_{q^2}$
we may conjugate a Singer element to the multiplication action $t\mapsto \beta t$ on $U_{q+1}$.
In this section we exploit this multiplicative parametrization to introduce diagonal-monomial orbit models in $\PG(3,K)$,
which will serve as the normal form for our later equivalence and descent arguments.

\begin{definition}\label{def:monomial-arc}
Let $K=\F_{q^2}$ and $n=q+1$. Let
$
U_n=\{t\in K^\times:\ t^n=1\},
$
and fix a generator $\beta\in U_n$ (so $U_n=\langle\beta\rangle$).
For $a\in(\Z/n\Z)^\times$  we define the \emph{cyclic monomial model}
\[
\mathcal{M}_a
:=\bigl\{[\,1:t:t^a:t^{a+1}\,]:\ t\in U_n\bigr\}\subset \PG(3,K),
\]
and the associated diagonal element
\[
S_a:=\mathrm{diag}(1,\beta,\beta^a,\beta^{a+1})\in\GL(4,K).
\]
Then $[S_a]\in\PGL(4,K)$ stabilizes $\mathcal{M}_a$ and acts on the parameter by
$t\mapsto \beta t$. In particular, $[S_a]$ has order $n$ and acts regularly
 on $\mathcal{M}_a$.
\end{definition}

The next lemma identifies $A_e$ over $K$ with
the monomial model $\mathcal M_{2^e}$ and shows that,
under this identification, any Singer element in the lifted subgroup $\varphi(\PGL(2,q))\le \mathrm{Stab}(A_e)$
is sent to the corresponding diagonal element.

\begin{lemma}\label{lem:Aeq-to-monomial}
Let $q=2^m$ with $m\ge 3$ and let $e$ satisfy $\gcd(e,m)=1$.
Then the arc $A_{e}\subset \PG(3,q)$ is projectively equivalent over $\F_{q^2}$ to $\mathcal{M}_{2^e}$.
Moreover, one can choose $P\in\PGL(4,K)$ with
$P(A_e)=\mathcal M_{2^e}$ such that,
for some Singer element $M\in\PGL(2,q)$,
writing $g=\varphi(M)\in\PGL(4,q)$, we have
\[
P\,g\,P^{-1}=[S_{2^e}]\in\PGL(4,K).
\]
\end{lemma}
\begin{proof}
Put $\sigma=2^e$ and $K=\F_{q^2}$. Since $x\mapsto x^\sigma$ is a Frobenius automorphism of $K$,
the map
\[
\psi_K:\PG(1,K)\to\PG(3,K),\qquad
\psi_K([x:y])=[\,x^{\sigma+1}:x^\sigma y:xy^\sigma:y^{\sigma+1}\,]
\]
is well-defined, and the same computation as in 
Lemma~\ref{lem:PGL2-lift} yields an equivariant homomorphism
\[
\varphi_K:\PGL(2,K)\to\PGL(4,K),\qquad [M]\longmapsto [M_{a,b,c,d}],
\]
satisfying $\varphi_K(X)\,\psi_K(\tau)=\psi_K(X\cdot\tau)$ for all $X\in\PGL(2,K)$ and $\tau\in\PG(1,K)$.
In particular, $\psi_K$ restricts to $\psi$ (see Lemma~\ref{lem:PGL2-lift}) on $\PG(1,q)$, so $A_e=\psi(\PG(1,q))=\psi_K(\PG(1,q))$.
Fix a Singer element $M\in\PGL(2,q)$ and set $g:=\varphi(M)\in\PGL(4,q)\le\PGL(4,K)$.
Apply Lemma~\ref{lem:phi-conjugates-singer} to $M$:
there exist a projectivity
$\phi_0\in\PGL(2,K)$ and an element
$\beta\in U_{q+1}$ of order $q+1$ such that
\[
\phi_0(\PG(1,q))=\{[u:1]:u\in U_{q+1}\}
\quad\text{and}\quad
\phi_0\,M\,\phi_0^{-1}=[\diag(\beta,1)]\in\PGL(2,K).
\]
Let $\iota\in\PGL(2,K)$ be the inversion $\iota([x:y])=[y:x]$, and put $\phi:=\iota\circ\phi_0$.
Then
\[
\phi(\PG(1,q))=\{[1:u]:u\in U_{q+1}\}
\quad\text{and}\quad
\phi\,M\,\phi^{-1}=\iota[\diag(\beta,1)]\iota^{-1}=[\diag(1,\beta)].
\]
Now set $P:=\varphi_K(\phi)\in\PGL(4,K)$. Using equivariance,
\begin{equation*}
\begin{split}
P(A_e)&=\varphi_K(\phi)\,\psi_K(\PG(1,q))
      =\psi_K(\phi(\PG(1,q)))\\
      &=\{\psi_K([1:u]):u\in U_{q+1}\}
      =\{[\,1:u:u^\sigma:u^{\sigma+1}\,]:u\in U_{q+1}\}
      =\mathcal M_\sigma=\mathcal M_{2^e},
\end{split}
\end{equation*}
proving the first claim in the lemma.
Finally, since $\varphi_K$ is a homomorphism
and $\varphi_K(M)=\varphi(M)=g$, we have
\[
P\,g\,P^{-1}
=\varphi_K(\phi)\,\varphi_K(M)\,\varphi_K(\phi)^{-1}
=\varphi_K(\phi\,M\,\phi^{-1})
=\varphi_K([\diag(1,\beta)]).
\]
Substituting $a=1,b=0,c=0,d=\beta$ into \eqref{eq:Mabcd-def} gives
\[
\varphi_K([\diag(1,\beta)])=
[\diag(1,\beta,\beta^\sigma,\beta^{\sigma+1})]=[S_\sigma]=[S_{2^e}],
\]
as required.
\end{proof}

% ------------------------------------------------------------

% ------------------------------------------------------------

The next lemma is the spectral bridge: projective equivalence forces conjugacy of the corresponding regular cyclic actions, hence equality of eigenvalue multisets, which translates into an affine relation among the exponent sets modulo $q+1$.

% ------------------------------------------------------------
% Lemma: projective equivalence to the model M_{2^e}
% forces affine equivalence of exponent sets
% ------------------------------------------------------------
\begin{lemma}
\label{lem:conj-implies-affine-exponents-correct}
Let $q=2^m$ with $m\ge 3$, let $n=q+1$, and let $K=\F_{q^2}$.
Fix $\beta\in K^\times$ of order $n$, and define the cyclic monomial models
$\mathcal M_a,\mathcal M_{2^e}\subset \PG(3,K)$ and diagonal elements
$S_a,S_{2^e}\in \GL(4,K)$ as in Definition~\ref{def:monomial-arc}, where
$a\in(\Z/n\Z)^\times$ and $e$ satisfies $\gcd(e,m)=1$.
Assume that there exists $P\in\PGL(4,K)$ such that
$
P(\mathcal M_a)=\mathcal M_{2^e}.
$
Then there exist $u\in(\Z/n\Z)^\times$ and $v\in \Z/n\Z$ such that
\[
\{0,1,a,a+1\}\ \equiv\ v+u\cdot\{0,1,2^e,2^e+1\}\pmod n
\]
as unordered subsets of $\Z/n\Z$.
\end{lemma}

\begin{proof}
Put $b:=2^e$.
By Definition~\ref{def:monomial-arc}, the cyclic groups $\langle [S_a]\rangle$ and
$\langle [S_b]\rangle$ have order $n$ and act regularly  
on $\mathcal M_a$ and $\mathcal M_b$, respectively.
Let
\[
g:=P\,[S_a]\,P^{-1}\in \PGL(4,K).
\]
Then $g$ stabilizes $\mathcal M_b$ and has order $n$;
moreover $g$ acts regularly on $\mathcal M_b$.
Let
$G_b:=\mathrm{Stab}_{\PGL(4,K)}(\mathcal M_b).$
We claim that there exist $Q\in G_b$ and $u\in(\Z/n\Z)^\times$ such that
\begin{equation}\label{eq:QgQ-1}
Q\,g\,Q^{-1} = [S_b]^u\qquad\text{in }\PGL(4,K).
\end{equation}
Assuming \eqref{eq:QgQ-1} for the moment, set $P_1:=Q P$.
Since $Q\in G_b$ stabilizes $\mathcal M_b$, we still have $P_1(\mathcal M_a)=\mathcal M_b$, and
we obtain an equality in $\PGL(4,K)$:
\begin{equation}\label{eq:conj-to-power-final}
P_1\,[S_a]\,P_1^{-1}=[S_b]^u.
\end{equation}
Choose the diagonal representatives in $\GL(4,K)$:
\[
\widetilde S_a=\mathrm{diag}(1,\beta,\beta^a,\beta^{a+1}),
\qquad
\widetilde S_b^u=\mathrm{diag}(1,\beta^u,\beta^{ub},\beta^{u(b+1)}).
\]
Equality \eqref{eq:conj-to-power-final} in $\PGL(4,K)$ means that there exists $\lambda\in K^\times$ such that
$
P_1\,\widetilde S_a\,P_1^{-1}=\lambda\,\widetilde S_b^u~ 
\text{in }\GL(4,K).
$
Conjugation preserves the multiset of eigenvalues, hence
\[
\{1,\beta,\beta^a,\beta^{a+1}\}
=
\{\lambda,\lambda\beta^u,\lambda\beta^{ub},\lambda\beta^{u(b+1)}\}
\]
as unordered multisets in $K^\times$.
In particular, $\lambda$ is one of the eigenvalues on the right, hence belongs to the cyclic group
$\langle\beta\rangle$ because the left-hand side lies in $\langle\beta\rangle$.
Write $\lambda=\beta^v$ for some $v\in\Z/n\Z$.
Taking exponents modulo $n$ yields
\[
\{0,1,a,a+1\}\equiv v+u\cdot\{0,1,b,b+1\}\pmod n,
\]
which is exactly the desired conclusion (recall $b=2^e$).
Thus the lemma follows once we establish \eqref{eq:QgQ-1}. We now prove \eqref{eq:QgQ-1}.
By Lemma~\ref{lem:Aeq-to-monomial}, there exists $T\in\PGL(4,K)$ such that
\begin{equation}\label{eq:T-model}
T(A_e)=\mathcal M_b
\qquad\text{and}\qquad
T\,h\,T^{-1}=[S_b]
\end{equation}
for some Singer element $h\in\PGL(4,q)$ lying in the lifted subgroup
$$\varphi(\PGL(2,q))\le \mathrm{Stab}_{\PGL(4,q)}(A_e)$$ from Lemma~\ref{lem:PGL2-lift}.
Set
\[
g_0:=T^{-1} g T\in \PGL(4,K),
\qquad
h_0:=T^{-1}[S_b]T=h\in \PGL(4,q).
\]
Then $g_0$ stabilizes $A_e$ and has order $n$; moreover, since $g$ acts regularly on $\mathcal M_b$
and $T$ is a bijection $A_e\to\mathcal M_b$, the element $g_0$ acts regularly on $A_e$.
Thus $g_0$ and $h_0$ are  Singer elements  in the stabilizer
\[
G_e:=\mathrm{Stab}_{\PGL(4,K)}(A_e).
\]

\smallskip
\noindent\textbf{Claim 3.1.}
\label{cl:stabK_equals_stabq}
We have
\[
\mathrm{Stab}_{\PGL(4,K)}(A_e)=\mathrm{Stab}_{\PGL(4,q)}(A_e).
\]

\smallskip\noindent
\smallskip\noindent
\emph{Proof of Claim 3.1.}
Let $\Fr:\PG(3,K)\to\PG(3,K)$ be the $q$-Frobenius map
$\Fr([x_0:\cdots:x_3])=[x_0^q:\cdots:x_3^q]$.
Since $A_e\subset \PG(3,q)$, every point of $A_e$ is fixed by $\Fr$.
Let $X\in\mathrm{Stab}_{\PGL(4,K)}(A_e)$. For any $P\in A_e$ we have
\[
(\Fr X \Fr^{-1})(P)=\Fr(X(\Fr^{-1}(P)))=\Fr(X(P))=X(P),
\]
because $X(P)\in A_e$ is also fixed by $\Fr$.
Hence $\Fr X \Fr^{-1}$ and $X$ agree on $A_e$.
Note that the set $A_e$ is a $(q+1)$-arc in $\PG(3,q)$.
In particular it contains a projective frame 
(indeed any five points of $A_e$ are in general position),
and a projectivity of $\PG(3,K)$ is uniquely determined by its action on a projective frame.
Therefore $\Fr X \Fr^{-1}=X$ in $\PGL(4,K)$.

Now choose a representative matrix $A\in\GL(4,K)$ for $X$.
Conjugation by $\Fr$ sends $[A]$ to $[A^{(q)}]$, where $A^{(q)}$ is obtained by raising
all entries of $A$ to the $q$th power.
Thus $\Fr X \Fr^{-1}=X$ implies $[A^{(q)}]=[A]$ in $\PGL(4,K)$, i.e.\ there exists $c\in K^\times$
such that $A^{(q)}=cA$. Applying $q$-Frobenius again yields
\[
A=A^{(q^2)}=(A^{(q)})^{(q)}=(cA)^{(q)}=c^qA^{(q)}=c^{q+1}A,
\]
hence $c^{q+1}=1$, so $c\in U_{q+1}$.
Since the map $K^\times\to U_{q+1}$, $\mu\mapsto \mu^{q-1}$ is surjective,
we may choose $\mu\in K^\times$ with $\mu^{q-1}=c$.
Set $B:=\mu^{-1}A$. Then
\[
B^{(q)}=\mu^{-q}A^{(q)}=\mu^{-q}cA=\mu^{-q}\mu^{q-1}A=\mu^{-1}A=B,
\]
so $B$ has all entries in $\F_q$, i.e.\ $B\in\GL(4,q)$.
Hence $X=[B]\in\PGL(4,q)$, proving
$$
\mathrm{Stab}_{\PGL(4,K)}(A_e)\subseteq \mathrm{Stab}_{\PGL(4,q)}(A_e).
$$
The reverse inclusion is obvious, so the stabilizers are equal.
\hfill$\square$

\smallskip
By Claim~3.1 we have
$g_0,h_0\in\mathrm{Stab}_{\PGL(4,q)}(A_e)$.
By Lemma~\ref{lem:PGL2-lift}(3), $\mathrm{Stab}_{\PGL(4,q)}(A_e)=\varphi(\PGL(2,q))$,
so $g_0,h_0\in\varphi(\PGL(2,q))$.
Hence there exist Singer elements $M_1,M_0\in\PGL(2,q)$ such that
\[
g_0=\varphi(M_1),
\qquad
h_0=\varphi(M_0).
\]

\smallskip
\smallskip\noindent\textbf{Claim 3.2.}
If $M_1,M_0\in\PGL(2,q)$ are Singer elements (order $q+1$), then there exist
$R\in\PGL(2,q)$ and $u\in(\Z/n\Z)^\times$ such that
\[
R M_1 R^{-1}=M_0^u.
\]

\smallskip\noindent
\emph{Proof of Claim 3.2.}
Let $K=\F_{q^2}$. By Lemma~\ref{lem:phi-conjugates-singer}(a),
each Singer element of $\PGL(2,q)$ has exactly two fixed points in
$\PG(1,K)\setminus\PG(1,q)$, forming a Frobenius-conjugate pair.
Let $\{\gamma_1,\gamma_1^q\}$ and $\{\gamma_0,\gamma_0^q\}$ be the fixed pairs
of $M_1$ and $M_0$, respectively.
Since $\PGL(2,q)$ acts transitively on $\PG(1,K)\setminus\PG(1,q)$,
there exists $R\in\PGL(2,q)$ with $R(\gamma_1)=\gamma_0$.
Because $R$ has coefficients in $\F_q$, it commutes with Frobenius, hence
$R(\gamma_1^q)=R(\gamma_1)^q=\gamma_0^q$.
Therefore $R M_1 R^{-1}$ and $M_0$ fix the same Frobenius pair
$\{\gamma_0,\gamma_0^q\}$, so both lie in the stabilizer of that pair.

But the stabilizer of $\{\gamma_0,\gamma_0^q\}$ in $\PGL(2,q)$ is a cyclic group
of order $q+1$ (conjugate to $U_{q+1}$ via Lemma~\ref{lem:phi-conjugates-singer}(b)--(c)).
Hence there exists $u$ with $\gcd(u,q+1)=1$ such that
$R M_1 R^{-1}=M_0^u$.
\qed

\smallskip
Apply Claim~3.2 to obtain $R\in\PGL(2,q)$ and $u\in(\Z/n\Z)^\times$ with
$R M_1 R^{-1}=M_0^u$.
Let $Q_0:=\varphi(R)\in \varphi(\PGL(2,q))\le G_e$. Then
\[
Q_0\,g_0\,Q_0^{-1}
=\varphi(R)\,\varphi(M_1)\,\varphi(R)^{-1}
=\varphi(R M_1 R^{-1})
=\varphi(M_0^u)
=\varphi(M_0)^u
=h_0^u.
\]
Conjugating by $T$ and using $h_0=T^{-1}[S_b]T$ yields
\[
(TQ_0T^{-1})\,g\,(TQ_0T^{-1})^{-1}=[S_b]^u.
\]
Thus \eqref{eq:QgQ-1} holds with $Q:=TQ_0T^{-1}\in G_b$.
This completes the proof of the lemma.
\end{proof}

% ------------------------------------------------------------
We are now in position to complete the passage from projective equivalence of cyclic monomial models
to a concrete arithmetic restriction on the exponent parameter.
The previous lemma reduces projective equivalence to an affine equivalence of the exponent sets
$\{0,1,a,a+1\}$ modulo $n=q+1$.
The following short combinatorial lemma is the final ingredient: 
it shows that for four-point sets of the
special form $\{0,1,a,a+1\}$, affine equivalence forces 
$a$ to be congruent to $\pm b$ or $\pm b^{-1}$.
This will close the proof of the main geometric theorem.

% ------------------------------------------------------------

\begin{lemma}\label{lem:four-point-affine}
Let $n$ be an odd integer with $n\ge 9$.
Let $b\in(\Z/n\Z)^\times$ and $a\in \Z/n\Z$.
Assume there exist $u\in(\Z/n\Z)^\times$ and $v\in\Z/n\Z$ such that
\[
\{0,1,a,a+1\} \equiv v + u\cdot\{0,1,b,b+1\}\pmod n
\]
as unordered sets. Then
\[
a \equiv \pm b \pmod n
\qquad\text{or}\qquad
a \equiv \pm b^{-1}\pmod n.
\]
\end{lemma}
\begin{proof}
Let $S_x = \{0, 1, x, x+1\} \subset \Z/n\Z$.
Since $\{0,1,a,a+1\}$ and $\{0,1,b,b+1\}$ are $4$-element 
sets, we have $a, b \not\equiv 0, \pm 1 \pmod n$.
Let $D(S)$ denote the set of differences between 
distinct elements of a set $S$.
Consider the  repeated differences  in $S_x$, 
i.e. differences that occur more than once.
The pairs $\{0,1\}$ and $\{x, x+1\}$ both yield difference $\pm 1$.
The pairs $\{0,x\}$ and $\{1, x+1\}$ both yield difference $\pm x$.
The remaining pairs $\{0, x+1\}$ and $\{1, x\}$ yield $\pm(x+1)$ and $\pm(x-1)$, which are unique (since $n \ge 9$ and $x \not\equiv 0, \pm 1$).
Thus, the set of repeated differences for $S_x$ is $R_x = \{\pm 1, \pm x\}$.
The affine relation $S_a = v + u S_b$ implies that the differences in $S_a$ are exactly $u$ times the differences in $S_b$.
Consequently, the set of repeated differences must satisfy:
$
R_a = u \cdot R_b.
$
Substituting the explicit forms:
\[
\{\pm 1, \pm a\} = \{\pm u, \pm ub\}.
\]
Since $1 \in \{\pm 1, \pm a\}$, the element $1$ must appear on the right-hand side. There are two cases:
\begin{enumerate}[leftmargin=2em]
    \item \textbf{Case 1:} $1 \in \{\pm u\}$.
    Then $u \equiv \pm 1 \pmod n$.
    The set equation becomes $\{\pm 1, \pm a\} = \{\pm 1, \pm b\}$.
    Removing $\pm 1$ from both sides leaves $\{\pm a\} = \{\pm b\}$, which implies $a \equiv \pm b \pmod n$.

    \item \textbf{Case 2:} $1 \in \{\pm ub\}$.
    Then $u \equiv \pm b^{-1} \pmod n$.
    Substituting $u = \epsilon b^{-1}$ (where $\epsilon \in \{1, -1\}$) into the RHS:
    \[
    \{\pm u, \pm ub\} = \{\pm b^{-1}, \pm 1\}.
    \]
    The set equation becomes $\{\pm 1, \pm a\} = \{\pm 1, \pm b^{-1}\}$.
    Removing $\pm 1$ leaves $\{\pm a\} = \{\pm b^{-1}\}$, which implies $a \equiv \pm b^{-1} \pmod n$.
\end{enumerate}
This covers all possibilities and completes the proof.
\end{proof}
% ============================================================

% ============================================================

We now combine the diagonalization of $A_e$ (Lemma~\ref{lem:Aeq-to-monomial}),
the spectral rigidity for cyclic monomial models (Lemma~\ref{lem:conj-implies-affine-exponents-correct}),
and the final four-point combinatorial constraint (Lemma~\ref{lem:four-point-affine})
to obtain the main geometric classification theorem.
It characterizes exactly when a cyclic monomial model $\mathcal{M}_a\subset \PG(3,\F_{q^2})$
is $\F_{q^2}$-projectively equivalent to a $(q+1)$-arc contained in the subgeometry $\PG(3,q)$.

\begin{theorem}
\label{thm:geom-main}
Let $q=2^m$ with $m\ge 3$ and put $n=q+1$.
Let $a\in(\Z/n\Z)^\times$ and consider the cyclic monomial point set $\mathcal{M}_a\subset \PG(3,\F_{q^2})$.
The following are equivalent:
\begin{enumerate}[label=\textup{(\roman*)}, leftmargin=2.6em]
\item $\mathcal{M}_a$ is projectively equivalent (over $\F_{q^2}$) to a $(q+1)$-arc contained in $\PG(3,q)$.
\item $\mathcal{M}_a$ is projectively equivalent (over $\F_{q^2}$) to $A_{e}$ for some $e$ with $\gcd(e,m)=1$.
\item There exists $e$ with $\gcd(e,m)=1$ such that $a\equiv \pm 2^e \pmod n$.
\end{enumerate}
\end{theorem}

\begin{proof}
\textbf{(i)$\Rightarrow$(ii).}
Assume $\mathcal{M}_a$ is projectively equivalent to a $(q+1)$-arc $\mathcal{A}\subset \PG(3,q)$.
By Theorem~\ref{thm:CasseGlynn}, $\mathcal{A}$ is projectively equivalent (over $\F_q$) to some $A_{e}$ with $\gcd(e,m)=1$.
Extending scalars to $\F_{q^2}$ shows $\mathcal{M}_a$ is projectively equivalent to $A_{e}$.

\textbf{(ii)$\Rightarrow$(iii).}
Assume $\mathcal{M}_a\simeq A_{e}$ over $\F_{q^2}$.
By Lemma~\ref{lem:Aeq-to-monomial}, $A_{e}\simeq \mathcal{M}_{2^e}$ over $\F_{q^2}$.
Hence $\mathcal{M}_a\simeq \mathcal{M}_{2^e}$.
Applying Lemma~\ref{lem:conj-implies-affine-exponents-correct}
with $b=2^e$ yields an affine equivalence
\[
\{0,1,a,a+1\}\equiv v+u\cdot\{0,1,2^e,2^e+1\}\pmod n.
\]
Lemma~\ref{lem:four-point-affine} then implies
\[
a\equiv \pm 2^e \pmod n \qquad\text{or}\qquad a\equiv \pm (2^e)^{-1}\pmod n.
\]
Since $n=2^m+1$ we have $2^m\equiv -1\pmod n$ and thus $(2^e)^{-1}\equiv -2^{m-e}\pmod n$.
Because $\gcd(e,m)=1$ implies $\gcd(m-e,m)=1$, the second possibility also has the form $a\equiv \pm 2^{e'}$ with $\gcd(e',m)=1$.
This proves (iii).

\textbf{(iii)$\Rightarrow$(i).}
Assume that $a\equiv \pm 2^e \pmod n$ with $\gcd(e,m)=1$.
If $a\equiv 2^e\pmod n$, then $\mathcal{M}_a=\mathcal{M}_{2^e}$.
If $a\equiv -2^e\pmod n$, then for $t\in U_n$ we have $t^a=t^{-2^e}$.
Reparameterizing by $t\mapsto t^{-1}$ gives
\[
\mathcal{M}_a
=\bigl\{[\,1:t^{-1}:t^{2^e}:t^{2^e-1}\,]:t\in U_n\bigr\}
=\bigl\{[\,t:1:t^{2^e+1}:t^{2^e}\,]:t\in U_n\bigr\}.
\]
Applying the coordinate permutation $[x_0:x_1:x_2:x_3]\mapsto [x_1:x_0:x_3:x_2]$
sends this set to $\mathcal{M}_{2^e}$.

In either case, $\mathcal{M}_a$ is projectively equivalent over $\F_{q^2}$ to $\mathcal{M}_{2^e}$.
By Lemma~\ref{lem:Aeq-to-monomial}, $\mathcal{M}_{2^e}$ is projectively equivalent over $\F_{q^2}$
to $A_e\subset \PG(3,q)$, which is a $(q+1)$-arc. Hence $\mathcal{M}_a$ is projectively equivalent
to a $(q+1)$-arc contained in $\PG(3,q)$, proving \textup{(i)}.
\end{proof}

In the following
corollary,
we refine the geometric classification by
keeping track of a chosen regular cyclic subgroup,
leading to a normal form and an exact count of
$K$-projective equivalence classes of pairs $(\mathcal{A},C)$.

\begin{corollary}
\label{cor:regular-cyclic-classification}
Let $q=2^m$ with $m\ge 3$, put $n=q+1$, and let $K=\F_{q^2}$.
Let $\mathcal{A}\subset \PG(3,q)$ be a $(q+1)$-arc and let
$C\le \PGL(4,q)$ be a cyclic subgroup of order $n$ acting regularly on $\mathcal{A}$.
Then there exist an integer $e$ with $\gcd(e,m)=1$ and a projectivity $P\in\PGL(4,K)$ such that
\[
P(\mathcal{A})=\mathcal{M}_{2^e}
\qquad\text{and}\qquad
P C P^{-1}=\langle [S_{2^e}]\rangle,
\]
where $\mathcal{M}_{2^e}$ and $[S_{2^e}]$ are as in Definition~\ref{def:monomial-arc}.
Moreover, if also $P'(\mathcal{A})=\mathcal{M}_{2^{e'}}$ and $P' C (P')^{-1}=\langle [S_{2^{e'}}]\rangle$
for some $\gcd(e',m)=1$, then $e'\in\{e,m-e\}$.
Consequently, up to projective equivalence over $K$, the pairs $(\mathcal{A},C)$ fall into exactly
$\varphi(m)/2$ equivalence classes.
\end{corollary}

\begin{proof}
By Theorem~\ref{thm:CasseGlynn}, there exist an integer $e$ with $\gcd(e,m)=1$
(in the standard range $1\le e\le m-1$) and $P_0\in\PGL(4,q)$ such that
$P_0(\mathcal{A})=A_e$. Conjugating the pair $(\mathcal{A},C)$ by $P_0$, we may assume
$\mathcal{A}=A_e$ and $C\le \PGL(4,q)$ still has order $n$ and acts regularly on $A_e$.
By Lemma~\ref{lem:PGL2-lift}(3),
\[
\mathrm{Stab}_{\PGL(4,q)}(A_e)=\varphi(\PGL(2,q))\simeq \PGL(2,q),
\]
so in particular $C\le \varphi(\PGL(2,q))$.
Choose a generator $g$ of $C$. Since $\varphi$ is an isomorphism onto
$\mathrm{Stab}_{\PGL(4,q)}(A_e)$, there exists $M\in\PGL(2,q)$ with
$\varphi(M)=g$.
Thus $M$ is a Singer element in $\PGL(2,q)$.
Applying Lemma~\ref{lem:Aeq-to-monomial} to this Singer element $M$, we obtain
a projectivity $P_1\in\PGL(4,K)$ such that
\[
P_1(A_e)=\mathcal{M}_{2^e}
\qquad\text{and}\qquad
P_1\,\varphi(M)\,P_1^{-1}=[S_{2^e}].
\]
Since $C=\langle g\rangle=\langle \varphi(M)\rangle$, it follows that
$
P_1 C P_1^{-1}=\langle [S_{2^e}]\rangle.
$
Undoing the initial conjugation, with $P:=P_1P_0\in\PGL(4,K)$ we have
\[
P(\mathcal{A})=\mathcal{M}_{2^e}
\qquad\text{and}\qquad
P C P^{-1}=\langle [S_{2^e}]\rangle,
\]
which proves the first assertion of the corollary.
Let $a\in(\Z/n\Z)^\times$.
Consider the coordinate permutations in $\PGL(4,K)$
\[
\Pi_{12}([x_0:x_1:x_2:x_3])=[x_0:x_2:x_1:x_3],\qquad
\Pi_{01,23}([x_0:x_1:x_2:x_3])=[x_1:x_0:x_3:x_2].
\]
\emph{(i) Action on point sets.}
For $t\in U_n$ we have
\[
\Pi_{12}\bigl([1:t:t^a:t^{a+1}]\bigr)=[1:t^a:t:t^{a+1}],
\]
and setting $s:=t^a$ (a permutation of $U_n$ since $\gcd(a,n)=1$) yields
$[1:s:s^{a^{-1}}:s^{a^{-1}+1}]$. Hence $\Pi_{12}(\mathcal{M}_a)=\mathcal{M}_{a^{-1}}$.
Likewise,
\[
\Pi_{01,23}\bigl([1:t:t^a:t^{a+1}]\bigr)=[t:1:t^{a+1}:t^a]\sim [1:t^{-1}:t^a:t^{a-1}],
\]
and with $u:=t^{-1}$ we obtain $[1:u:u^{-a}:u^{-a+1}]$, so
$\Pi_{01,23}(\mathcal{M}_a)=\mathcal{M}_{-a}$.
Therefore
\[
(\Pi_{01,23}\circ \Pi_{12})(\mathcal{M}_a)=\mathcal{M}_{-a^{-1}}.
\]

\emph{(ii) Action on the cyclic groups.}
A direct conjugation computation in $\GL(4,K)$ gives
\[
\Pi_{12} S_a \Pi_{12}^{-1}=\diag(1,\beta^a,\beta,\beta^{a+1})=S_{a^{-1}}^{\,a},
\]
hence
$\Pi_{12}\langle [S_a]\rangle\Pi_{12}^{-1}=\langle [S_{a^{-1}}]\rangle$.
Similarly,
\[
\Pi_{01,23} S_a \Pi_{01,23}^{-1}=\diag(\beta,1,\beta^{a+1},\beta^a)
\sim \diag(1,\beta^{-1},\beta^a,\beta^{a-1})=S_{-a}^{-1}
\]
in $\PGL(4,K)$, hence
$\Pi_{01,23}\langle [S_a]\rangle\Pi_{01,23}^{-1}=\langle [S_{-a}]\rangle$.
Consequently,
\[
(\Pi_{01,23}\circ\Pi_{12})\ \text{sends}\
(\mathcal{M}_a,\langle [S_a]\rangle)\ \text{to}\
(\mathcal{M}_{-a^{-1}},\langle [S_{-a^{-1}}]\rangle).
\]
Now take $a=2^e$ (so $\gcd(a,n)=1$ since $n$ is odd).
Because $n=2^m+1$ and $2^e\cdot 2^{m-e}=2^m\equiv -1\pmod n$, we have
$-a^{-1}\equiv 2^{m-e}\pmod n$.
Since $\mathcal{M}_b$ and $[S_b]$ depend only on $b\bmod n$ (because $\beta^n=1$ and $t^n=1$ on $U_n$),
we obtain an explicit equivalence over $K$:
\[
(\mathcal{M}_{2^e},\langle [S_{2^e}]\rangle)\ \simeq\
(\mathcal{M}_{2^{m-e}},\langle [S_{2^{m-e}}]\rangle).
\]
Assume also that for some $e'$ with $\gcd(e',m)=1$ (and again $1\le e'\le m-1$) there exists
$P'\in\PGL(4,K)$ such that
\[
P'(\mathcal{A})=\mathcal{M}_{2^{e'}}
\qquad\text{and}\qquad
P' C (P')^{-1}=\langle [S_{2^{e'}}]\rangle.
\]
Let $T:=P(P')^{-1}\in\PGL(4,K)$. Then
$
T(\mathcal{M}_{2^{e'}})=\mathcal{M}_{2^e}.
$
Applying Lemma~\ref{lem:conj-implies-affine-exponents-correct} with $a=2^{e'}$ and $b=2^e$
yields $u\in(\Z/n\Z)^\times$ and $v\in\Z/n\Z$ such that
\[
\{0,1,2^{e'},2^{e'}+1\}\equiv v+u\cdot\{0,1,2^e,2^e+1\}\pmod n.
\]
By Lemma~\ref{lem:four-point-affine} we deduce
\[
2^{e'}\equiv \pm 2^e \pmod n
\qquad\text{or}\qquad
2^{e'}\equiv \pm (2^e)^{-1}\pmod n.
\]
We now exclude the negative signs using $n=2^m+1$ and $1\le e,e'\le m-1$.
If $2^{e'}\equiv -2^e\pmod n$, then $n\mid (2^{e'}+2^e)$, but
\[
0<2^{e'}+2^e\le 2^{m-1}+2^{m-1}=2^m<n,
\]
a contradiction. Likewise, if $2^{e'}\equiv -2^{m-e}\pmod n$, then
$n\mid (2^{e'}+2^{m-e})$ but
\[
0<2^{e'}+2^{m-e}\le 2^{m-1}+2^{m-1}=2^m<n,
\]
again a contradiction. Finally, since $2^e\cdot 2^{m-e}=2^m\equiv -1\pmod n$, we have
$(2^e)^{-1}\equiv -2^{m-e}\pmod n$, so the second alternative is equivalent to
$2^{e'}\equiv \pm 2^{m-e}\pmod n$.
With the negative sign excluded, we conclude
\[
2^{e'}\equiv 2^e \pmod n
\qquad\text{or}\qquad
2^{e'}\equiv 2^{m-e}\pmod n.
\]
Because $0<2^{e},2^{e'},2^{m-e}<n$, these congruences are equalities in $\Z$, hence
$e'=e$ or $e'=m-e$. This proves the stated uniqueness.

There are exactly $\varphi(m)$ integers $e$
with $1\le e\le m-1$ and $\gcd(e,m)=1$.
We have shown that $e$ and $m-e$ yield equivalent pairs over $K$,
and  no further
identifications occur. The involution $e\mapsto m-e$ has no fixed points on this set
(for if $e=m-e$, then $m=2e$ and $\gcd(e,m)=1$ would force $m=2$, impossible since $m\ge 3$).
Hence each orbit has size $2$, and the number of equivalence classes is $\varphi(m)/2$.
\end{proof}

% ------------------------------------------------------------
\section{An MDS criterion for the BCH family
$\mathcal{C}_{(q,q+1,3,h)}$}
\label{MDS}
We now turn to the coding-theoretic application, translating the parity--check geometry of the BCH family into the monomial models above.
In this section we apply the geometric classification
of cyclic monomial models Theorem \ref{thm:geom-main}
 to a family of BCH codes
$\mathcal{C}_{(q,q+1,3,h)}$ studied by Xu et al.~\cite{Xu}.
We briefly recall the definition in a form suited to our geometric viewpoint.

\begin{definition}
\label{def:BCH-Cqq13h}
Let $q$ be a prime power and put $n:=q+1$.
Fix a primitive $n$th root of unity $\beta\in \F_{q^2}^\times$.
For $s\in\Z$, let $g_s(x)\in\F_q[x]$ denote the minimal polynomial of $\beta^s$ over $\F_q$.
The BCH code $\mathcal{C}_{(q,q+1,3,h)}$ is the cyclic code of length $n$ over $\F_q$ with generator polynomial
\[
g(x)=\mathrm{lcm}\bigl(g_h(x),g_{h+1}(x)\bigr)\in \F_q[x],
\]
that is,
\[
\mathcal{C}_{(q,q+1,3,h)}=\langle g(x)\rangle \subseteq \F_q[x]/(x^{n}-1).
\]
Equivalently, $\mathcal{C}_{(q,q+1,3,h)}$ is the cyclic code whose defining zeros include $\beta^h$ and $\beta^{h+1}$.
\end{definition}

Xu et al.~\cite{Xu} analyzed the family $\mathcal{C}_{(q,q+1,3,h)}$ and determined the minimum distance in the cases $d=3$ and $d=4$.
They proved that for every prime power $q$ one has
\[
d(\mathcal{C}_{(q,q+1,3,h)})=3 \quad\Longleftrightarrow\quad \gcd(2h+1,q+1)>1.
\]
In particular, the regime relevant to MDS behaviour is the complementary case $\gcd(2h+1,q+1)=1$.
When $q$ is odd, Xu et al.~\cite{Xu} further showed that $\gcd(2h+1,q+1)=1$ forces $d=4$,
so the parameters of $\mathcal{C}_{(q,q+1,3,h)}$ are completely settled in the odd case and the code can never be MDS of parameters $[q+1,q-3,5]$.
The picture changes for even $q$.
Assume $q=2^m$ with $m\ge 3$ and impose the
usual non-degeneracy conditions
\[
h\notin\{0,q/2,q\},\qquad \gcd(2h+1,q+1)=1,
\]
so that $\deg g(x)=4$ and hence $\dim \mathcal{C}_{(q,q+1,3,h)}=q-3$.
Xu et al.~\cite{Xu} observed that in this even case the code must be either AMDS ($d=4$) or MDS ($d=5$),
and explicitly highlighted the remaining problem of deciding which of these two possibilities occurs.

Our approach resolves this AMDS/MDS dichotomy by linking $\mathcal{C}_{(q,q+1,3,h)}$ to the cyclic monomial models from Section~\ref{sec:cyclic-monomial}.
On the one hand, the cross-ratio criterion of~\cite{Xu} detects the AMDS case $d=4$ via a projective condition on four points of $U_{q+1}$.
On the other hand, the defining zeros $\beta^h,\beta^{h+1}$ naturally produce a monomial orbit in $\PG(3,\F_{q^2})$
(Proposition~\ref{prop:BCH-to-monomial}), so that the MDS question becomes a geometric descent problem for a cyclic monomial model,
which is decided by Theorem~\ref{thm:geom-main}.
We begin with the sufficiency direction, using the cross-ratio criterion to rule out $d=4$ when $2h+1$ is a Frobenius exponent modulo $q+1$.

\begin{lemma}\label{sufficient}
Let $q=2^m$ with $m\ge 3$ and put $n=q+1$. Let $h$ be an integer such that
\[
0\le h\le q,\qquad
h\notin\Bigl\{0,\tfrac q2,q\Bigr\},\qquad
\gcd(2h+1,n)=1.
\]
Assume that there exists an integer $e\ge 1$ with $\gcd(e,m)=1$ such that
\[
2h+1\equiv \pm 2^e \pmod n.
\]
Then the  BCH code $\mathcal{C}_{(q,q+1,3,h)}$ is an MDS code.
\end{lemma}
\begin{proof}
Write   $k:=2h+1$. Under the assumptions
$h\notin\{0,\tfrac q2,q\}$ and $\gcd(k,n)=1$, it follows from
\cite[Theorem~1]{Xu} that the minimum distance $d$ of
$\mathcal{C}_{(q,q+1,3,h)}$ satisfies $d\ge 4$.
Moreover, since $q\equiv -1\pmod n$ and $h\neq0$,
the $q$-cyclotomic cosets of $h$ and
$h+1$ modulo $n$ are $\{h,-h\}$ and $\{h+1,-h-1\}$, each of size $2$, and
they are disjoint when
$h\notin \{0,q/2, q\}$. Hence the generator polynomial has
degree $4$ and $\dim \mathcal{C}_{(q,q+1,3,h)}=n-4=q-3$, so by the Singleton
bound \cite[Theorem~2.4.1]{Huffman} we have $d\le n-(n-4)+1=5$.
Therefore $d\in\{4,5\}$, and it suffices to show that $d\neq 4$.

Let $U_{q+1}$ be the subgroup of $\F_{q^2}^\times$ consisting of all
$(q+1)$th roots of unity, so $|U_{q+1}|=q+1$.
Fix $\beta\in\F_{q^2}$ of order $q+1$, so $U_{q+1}=\langle\beta\rangle$.
By \cite[Theorem~2]{Xu}, under $\gcd(k,n)=1$ we have $d=4$ if and only if
there exist four pairwise distinct elements $x,y,z,w\in U_{q+1}$ such that
\[
\frac{E(x,z)}{E(x,w)}=\frac{E(y,z)}{E(y,w)},
\qquad
E(a,b):=\frac{a^k-b^k}{a-b}.
\tag{$\ast$}
\]
A routine algebraic manipulation shows that, for pairwise distinct
$x,y,z,w$, the relation $(\ast)$ is equivalent to the cross-ratio identity
\begin{equation}\label{eq:CR-invariance}
\mathrm{CR}(x,y;z,w)=\mathrm{CR}(x^k,y^k;z^k,w^k),
\end{equation}
where
\[
\mathrm{CR}(x,y;z,w):=\frac{(x-z)(y-w)}{(x-w)(y-z)}.
\]
Hence $d=4$ is equivalent to the existence of pairwise distinct
$x,y,z,w\in U_{q+1}$ satisfying \eqref{eq:CR-invariance}.
Now let $x,y,z,w\in U_{q+1}$ be pairwise distinct and set
$t:=\mathrm{CR}(x,y;z,w)\in\F_{q^2}$.
By Lemma \ref{lem:CR-in-Fq}, we have
\begin{equation}\label{eq:t-not-01}
t=\mathrm{CR}(x,y;z,w)\in\F_q\setminus\{0,1\}.
\end{equation}
Assume now that $k\equiv \pm 2^e\pmod n$ with $\gcd(e,m)=1$.
For $x\in U_{q+1}$ we may reduce exponents modulo $n=q+1$, so
$x^k=x^{\pm 2^e}$. We claim that for all $x,y,z,w\in U_{q+1}$,
\[
\mathrm{CR}(x^k,y^k;z^k,w^k)=\mathrm{CR}(x,y;z,w)^{2^e}.
\]
Indeed, if $k\equiv 2^e$, then the map $a\mapsto a^{2^e}$ is the $2^e$-Frobenius
automorphism of $\F_{q^2}$, so
\[
\mathrm{CR}(x^{2^e},y^{2^e};z^{2^e},w^{2^e})
=\left(\mathrm{CR}(x,y;z,w)\right)^{2^e}.
\]
If $k\equiv -2^e$, then for $x\in U_{q+1}$ we have
$x^{-2^e}=(x^{-1})^{2^e}=(x^q)^{2^e}=(x^{2^e})^q$, hence
\[
\mathrm{CR}(x^k,y^k;z^k,w^k)
=\mathrm{CR}\bigl((x^{2^e})^q,(y^{2^e})^q;(z^{2^e})^q,(w^{2^e})^q\bigr)
=\mathrm{CR}(x^{2^e},y^{2^e};z^{2^e},w^{2^e})^q.
\]
Since $t=\mathrm{CR}(x,y;z,w)\in\F_q$,
we have $t^q=t$, and therefore again
$\mathrm{CR}(x^k,y^k;z^k,w^k)=t^{2^e}$.

Consequently, \eqref{eq:CR-invariance} is equivalent (for pairwise distinct
$x,y,z,w\in U_{q+1}$) to
\begin{equation}\label{eq:t-fixed}
t=t^{2^e},
\qquad\text{where }t=\mathrm{CR}(x,y;z,w)\in\F_q\setminus\{0,1\}.
\end{equation}
The fixed field of the Frobenius map
$t\mapsto t^{2^e}$ on $\F_{2^m}$ is $\F_{2}$ since
$\gcd(e,m)=1$.
Hence \eqref{eq:t-fixed} has no solution with $t\in\F_q\setminus\{0,1\}$.
Therefore \eqref{eq:CR-invariance} has
no solution in pairwise distinct $x,y,z,w\in U_{q+1}$, and thus $d\neq 4$.
Since $d\ge 4$, we conclude $d=5$, so $\mathcal{C}_{(q,q+1,3,h)}$ is MDS.
\end{proof}

We now connect $\mathcal{C}_{(q,q+1,3,h)}$ to a cyclic monomial model.

\begin{proposition}\label{prop:BCH-to-monomial}
Let $q=2^m$ with $m\ge 3$ and $n=q+1$. Fix $\beta\in \F_{q^2}^\times$ of order $n$.
Let $h$ be an integer and put $k=2h+1$.
Consider the $4\times n$ matrix over $\F_{q^2}$ whose $j$-th column ($j=0,\dots,n-1$) is
\[
c(\beta^j)=\begin{pmatrix}
(\beta^j)^h\\
(\beta^j)^{h+1}\\
(\beta^j)^{-h}\\
(\beta^j)^{-h-1}
\end{pmatrix}.
\]
Then, up to projective scaling of each column and a permutation of coordinates, the corresponding point set in $\PG(3,\F_{q^2})$
is exactly the cyclic monomial set $\mathcal{M}_{k}$:
\[
\bigl\{[\,1:t:t^k:t^{k+1}\,]: t\in U_{q+1}\bigr\}.
\]
Moreover, the cyclic shift $j\mapsto j+1$ acts as the diagonal Singer element $[S_k]=\mathrm{diag}(1,\beta,\beta^k,\beta^{k+1})$ in a suitable basis.
\end{proposition}

\begin{proof}
Let $t=\beta^j\in U_{q+1}$. The column is $c(t)=(t^h,t^{h+1},t^{-h},t^{-h-1})^T$.
Multiply this column by the nonzero scalar $t^{h+1}$ (allowed in projective space) to obtain
\[
t^{h+1}c(t)=(t^{2h+1},t^{2h+2},t,1)^T=(t^{k},t^{k+1},t,1)^T.
\]
Permuting coordinates to the order $(1,t,t^k,t^{k+1})$ yields the point
$[\,1:t:t^k:t^{k+1}\,]$.
As $t$ runs over $\langle\beta\rangle=U_{q+1}$, these points form precisely $\mathcal{M}_k$.
The action $j\mapsto j+1$ corresponds to $t\mapsto \beta t$ on the parameter.
Under the above normalization, this is realized by the diagonal matrix $S_k$ of Definition~\ref{def:monomial-arc}.
\end{proof}

We can now combine the two ingredients above.
Lemma~\ref{sufficient} gives a clean sufficient condition for the MDS case by excluding $d=4$ via the cross-ratio criterion,
while Proposition~\ref{prop:BCH-to-monomial} identifies the geometric object naturally attached to $\mathcal{C}_{(q,q+1,3,h)}$ over $K=\F_{q^2}$.
Applying the geometric classification theorem (Theorem~\ref{thm:geom-main}) to this object yields an exact MDS criterion.

\begin{theorem}
\label{thm:main-coding}
Let $q=2^m$ with $m\ge 3$ and let $n=q+1$. Let $h$ be an integer with
\[
0\le h\le q,\qquad h\notin\Bigl\{0,\tfrac q2,q\Bigr\},\qquad \gcd(2h+1,n)=1.
\]
Then the BCH code $\mathcal{C}_{(q,q+1,3,h)}$ is MDS  if and only if
there exists an integer $e$ with $\gcd(e,m)=1$ such that
\[
2h+1 \equiv \pm 2^e \pmod n.
\]
\end{theorem}
\begin{proof}
The ``if'' direction is Lemma~\ref{sufficient}.
For the ``only if'' direction, assume $\mathcal{C}_{(q,q+1,3,h)}$ is MDS.
Then $d=5$ and $\dim \mathcal{C}=q-3$, so
 there exists a parity--check matrix
$H\in\F_q^{4\times (q+1)}$ whose
column point set $\mathcal{P}(H)\subset \PG(3,q)$ is a $(q+1)$-arc.
On the other hand, Proposition~\ref{prop:BCH-to-monomial} provides
a $4\times (q+1)$ parity--check matrix
$\widetilde H$ over $K=\F_{q^2}$ for the same
code whose column point set in $\PG(3,K)$ is (up to projective scaling
and a permutation of coordinates) exactly the cyclic monomial model
$\mathcal{M}_k$ with $k=2h+1$.
Since $H$ and $\widetilde H$ are both full-rank parity--check matrices for $\mathcal{C}_{(q,q+1,3,h)}$,
their row spaces over $K$ coincide, hence there exists $A\in\GL(4,K)$ such that $\widetilde H=AH$.
Therefore $\mathcal{M}_k$ is $K$-projectively equivalent to the $(q+1)$-arc $\mathcal{P}(H)\subset \PG(3,q)$.

Applying Theorem~\ref{thm:geom-main} to $\mathcal{M}_k$ yields
$k\equiv \pm 2^e\pmod{q+1}$ for some $e$ with $\gcd(e,m)=1$.
Equivalently, $2h+1\equiv \pm 2^e\pmod{q+1}$.
\end{proof}

% ============================================================

Theorem~\ref{thm:main-coding} shows that the MDS condition is extremely restrictive on $h$ modulo $q+1$,
and in particular it allows only finitely many congruence classes (depending only on $m$).
The following corollary makes this explicit.
\begin{corollary}\label{cor:count}
Let $q=2^m$ with $m\ge 3$ and $n=q+1$. The set of congruence classes $h\in\Z/n\Z$ for which
$\mathcal{C}_{(q,q+1,3,h)}$ is MDS (under the standing hypotheses of Theorem~\ref{thm:main-coding})
is contained in
\[
\left\{\frac{\pm 2^e-1}{2}\bmod n : 1\le e\le m-1,\ \gcd(e,m)=1\right\}.
\]
In particular, there are at most $2\varphi(m)$ such congruence classes.
\end{corollary}

\begin{proof}
This is immediate from Theorem~\ref{thm:main-coding}.
The bound $2\varphi(m)$ counts the choices of $e$ coprime to $m$ together with the two signs; different
choices may coincide modulo $n$ for special $m$.
\end{proof}

For completeness, we spell out the criterion in a couple of small cases to illustrate the congruence condition
$2h+1\equiv \pm 2^e\pmod{q+1}$ and the symmetry $e\leftrightarrow m-e$.

\medskip
\begin{example}[$q=8$]
Here $m=3$, $n=9$, and $\gcd(e,3)=1$ gives $e\in\{1,2\}$.
The allowed values of $k=2h+1$ are
\[
k\equiv \pm 2,\ \pm 4 \pmod 9,
\]
so $k\in\{2,4,5,7\}$.
Since $2^{-1}\equiv 5\pmod 9$, we obtain
\[
h\equiv \frac{k-1}{2}\equiv 5(k-1)\pmod 9,
\]
yielding $h\in\{2,3,5,6\}$ modulo $9$ (excluding $h\in\{0,4,8\}$ as prescribed).
\end{example}

\medskip
\begin{example}[$q=16$]
Here $m=4$, $n=17$, and $\gcd(e,4)=1$ gives $e\in\{1,3\}$.
Thus $k\equiv \pm 2,\ \pm 8\pmod{17}$.
Since $2^{-1}\equiv 9\pmod{17}$, $h\equiv 9(k-1)\pmod{17}$.
\end{example}

\medskip
\begin{remark}
Because $n=q+1$ is odd, the map $k\mapsto -k$ corresponds to $h\mapsto (-k-1)/2$ and hence to a simple involution on $h$.
Also, $k^{-1}\equiv -2^{m-e}$ when $k=2^e$, reflecting the symmetry $e\leftrightarrow m-e$ in the geometric classification.
\end{remark}

\section{Spectral rigidity: a general principle and brief perspectives}
\label{sec:portable-spectral-rigidity}

The arguments in Sections~\ref{sec:cyclic-monomial}--\ref{MDS} illustrate a simple phenomenon:
once a regular cyclic action becomes diagonal over
a suitable extension field,
$K$-projective equivalence of the resulting
cyclic pairs is controlled by the eigenvalue data,
equivalently by exponent multisets modulo $n$.
We isolate this as a general rigidity principle
for diagonal regular cyclic pairs in $\PG(r,K)$,
and briefly indicate how it packages the mechanism used earlier.

\begin{definition}\label{def:diag-orbit}
Let $K$ be a field containing a primitive $n$th root of unity $\beta$, and put
$U_n=\langle\beta\rangle\subset K^\times$.
Fix $r\ge 1$ and a multiset $E=\{e_0,\ldots,e_r\}\subset \Z/n\Z$ of size $r+1$.
We define the diagonal orbit
\[
\mathcal M_E:=\bigl\{[t^{e_0}:\cdots:t^{e_r}]:\ t\in U_n\bigr\}\subset \PG(r,K),
\]
and the diagonal element
\[
S_E:=\diag(\beta^{e_0},\ldots,\beta^{e_r})\in \GL(r+1,K).
\]
We say that $E$ is \emph{regular} if $\gcd\bigl(n,\{e_i-e_0:\ 1\le i\le r\}\bigr)=1$.
Equivalently, the induced action of $[S_E]\in\PGL(r+1,K)$ on $\mathcal M_E$ is regular,
so $|\mathcal M_E|=n$.
In that case we call $(\mathcal M_E,[S_E])$ a \emph{diagonal regular cyclic orbit}.
\end{definition}

\begin{theorem}\label{thm:spectral-rigidity-practical}
Let $(\mathcal M_E,[S_E])$ and $(\mathcal M_F,[S_F])$ be diagonal regular cyclic orbits in $\PG(r,K)$.
Assume there exists $P\in\PGL(r+1,K)$ such that
\[
P(\mathcal M_E)=\mathcal M_F
\quad\text{and}\quad
P[S_E]P^{-1}\in \langle [S_F]\rangle .
\]
Then there exist $u\in(\Z/n\Z)^\times$ and $v\in\Z/n\Z$ such that
\[
E \equiv v+uF \pmod n
\]
as unordered multisets in $\Z/n\Z$.
\end{theorem}

\begin{proof}
Since $P[S_E]P^{-1}\in \langle [S_F]\rangle$ and both actions are regular of order $n$,
there exists $u\in(\Z/n\Z)^\times$ such that
\[
P[S_E]P^{-1}=[S_F]^u \qquad \text{in }\PGL(r+1,K).
\]
Choose $A\in\GL(r+1,K)$ with $[A]=P$. Then there exists $\lambda\in K^\times$ with
\[
A S_E A^{-1}=\lambda\, S_F^u \qquad \text{in }\GL(r+1,K).
\]
Comparing eigenvalue multisets gives
\[
\{\beta^{e_0},\ldots,\beta^{e_r}\}=\{\lambda\beta^{uf_0},\ldots,\lambda\beta^{uf_r}\}.
\]
As the left-hand side lies in $\langle\beta\rangle$, we have $\lambda\in\langle\beta\rangle$,
so $\lambda=\beta^v$ for some $v\in\Z/n\Z$.
Injectivity of $x\mapsto \beta^x$ on $\Z/n\Z$ yields
\[
\{e_0,\ldots,e_r\}\equiv \{v+uf_0,\ldots,v+uf_r\}\pmod n,
\]
i.e.\ $E\equiv v+uF$.
\end{proof}

\begin{remark}\label{rem:alignment}
The content of Theorem~\ref{thm:spectral-rigidity-practical} is the \emph{spectral step}:
once the cyclic generator has been aligned inside a fixed cyclic subgroup of the target model,
projective equivalence becomes an explicit affine congruence constraint on exponent data.
In our paper, the extra alignment hypothesis
$P[S_E]P^{-1}\in\langle[S_F]\rangle$
is obtained  from orbit equivalence alone 
 by a short stabilizer argument
(the ``alignment inside the stabilizer'' step in the proof of Lemma~\ref{lem:conj-implies-affine-exponents-correct}).
Outside the present $\PGL(2,q)$-stabilizer setting, such alignment need not be automatic and is therefore stated explicitly here.
\end{remark}

\medskip
In our concrete setting, Definition~\ref{def:monomial-arc} is the case $r=3$ with
$E=\{0,1,a,a+1\}$ (up to permuting coordinates).
Moreover, Lemma~\ref{lem:conj-implies-affine-exponents-correct} is exactly the corresponding application of
Theorem~\ref{thm:spectral-rigidity-practical} once the cyclic generators have been aligned inside the stabilizer
of the target model (cf.\ Remark~\ref{rem:alignment}).
Together with the short arithmetic analysis in Lemma~\ref{lem:four-point-affine}, this yields the descent criterion
(Theorem~\ref{thm:geom-main}) and hence the MDS test for $\mathcal{C}_{(q,q+1,3,h)}$
(Theorem~\ref{thm:main-coding}).

\medskip
More broadly, Theorem~\ref{thm:spectral-rigidity-practical} 
isolates the spectral step  as a black box:
whenever a regular cyclic symmetry becomes diagonal over an extension field and an alignment of cyclic generators is available,
projective equivalence reduces to checking an affine congruence relation on exponent data, $E\equiv v+uF\pmod n$.
The remaining difficulty is then combinatorial: solving this affine relation for the specific exponent patterns imposed by the geometry
(or by a code family).
We expect this viewpoint to be useful in other classification problems for cyclic configurations in projective spaces and for
cyclic/constacyclic MDS phenomena in larger codimension, where the alignment step may come either from an explicit stabilizer analysis
(as in the present paper) or from additional structural input.

\appendix
\section{Deferred proofs for Lemmas \ref{lem:PGL2-lift} and \ref{lem:phi-conjugates-singer}}
 \label{app:deferred-proofs}

\begin{proof}[Proof of Lemma~\ref{lem:PGL2-lift}]
First, \eqref{eq:psi-def} is well-defined on projective points:
if $[x:y]=[\lambda x:\lambda y]$ with $\lambda\in\F_q^\times$, then each
coordinate in \eqref{eq:psi-def} is multiplied by $\lambda^{\sigma+1}$, so
the resulting point of $\PG(3,q)$ is unchanged.

If $x\neq 0$, write $t=y/x\in\F_q$. Then
$
\psi([x:y])=[\,1:t:t^{\sigma}:t^{\sigma+1}\,]=P_t.
$
If $x=0$, then $[x:y]=[0:1]$ and $\psi([0:1])=[0:0:0:1]=P_\infty$.
Since $|\PG(1,q)|=q+1=|A_{e}|$ and $\psi$
is clearly surjective, it is a
bijection.

Fix $M=\begin{bmatrix}a&b\\ c&d\end{bmatrix}\in\mathrm{GL}(2,q)$ and $\tau=[x:y]\in\PG(1,q)$.
Put
$$[x':y']:=M\cdot[x:y]=[ax+by:cx+dy].$$
We compute $\psi([x':y'])$ in the homogeneous coordinates
$
(x'^{\sigma+1},\ x'^\sigma y',\ x' y'^\sigma,\ y'^{\sigma+1}).
$
Since $\sigma=2^e$ is a Frobenius power, we have
\begin{align*}
x'^{\sigma+1}
&=(ax+by)^{\sigma+1}=(ax+by)^\sigma(ax+by)\\
&=(a^\sigma x^\sigma+b^\sigma y^\sigma)(ax+by)\\
&=a^{\sigma+1}x^{\sigma+1}+a^\sigma b\,x^\sigma y
+ab^\sigma\,x y^\sigma + b^{\sigma+1}y^{\sigma+1},\\[1ex]
x'^\sigma y'
&=(ax+by)^\sigma(cx+dy)\\
&=a^\sigma c\,x^{\sigma+1}+a^\sigma d\,x^\sigma y
+b^\sigma c\,x y^\sigma + b^\sigma d\,y^{\sigma+1},\\[1ex]
x' y'^\sigma
&=(ax+by)(cx+dy)^\sigma\\
&=ac^\sigma\,x^{\sigma+1}+bc^\sigma\,x^\sigma y
+ad^\sigma\,x y^\sigma + bd^\sigma\,y^{\sigma+1},\\[1ex]
y'^{\sigma+1}
&=(cx+dy)^{\sigma+1}=(cx+dy)^\sigma(cx+dy)\\
&=(c^\sigma x^\sigma+d^\sigma y^\sigma)(cx+dy)\\
&=c^{\sigma+1}x^{\sigma+1}+c^\sigma d\,x^\sigma y
+cd^\sigma\,x y^\sigma + d^{\sigma+1}y^{\sigma+1}.
\end{align*}
These identities state exactly that the column vector representing
$\psi(M\cdot[x:y])$ is obtained by multiplying the column vector
representing $\psi([x:y])$ by the matrix $M_{a,b,c,d}$ from
\eqref{eq:Mabcd-def}. Hence \eqref{eq:equivariance} holds for all
$\tau\in\PG(1,q)$.
In particular, $\varphi(M)$ maps $A_{e}=\psi(\PG(1,q))$ to itself, so
$\varphi(\PGL(2,q))\le \mathrm{Stab}_{\PGL(4,q)}(A_{e})=G_e$.
The equivariance \eqref{eq:equivariance}
implies that $\varphi$ is a group
homomorphism: for $M_1,M_2\in\PGL(2,q)$
and all $\tau\in\PG(1,q)$,
\[
\varphi(M_1M_2)\psi(\tau)=\psi(M_1M_2\cdot\tau)
=\varphi(M_1)\psi(M_2\cdot\tau)
=\varphi(M_1)\varphi(M_2)\psi(\tau).
\]
Since $\psi(\PG(1,q))=A_{e}$ spans $\PG(3,q)$ and contains at least five
points in general position (because $|A_{e}|=q+1\ge 9$ and no four points
of an arc lie in a plane), a projectivity of $\PG(3,q)$ is uniquely
determined by its action on $A_{e}$. Thus
$\varphi(M_1M_2)=\varphi(M_1)\varphi(M_2)$ in $\PGL(4,q)$.

If $\varphi(M)$ is the identity in $\PGL(4,q)$, then by \eqref{eq:equivariance}
we have $\psi(M\cdot\tau)=\psi(\tau)$ for all $\tau\in\PG(1,q)$.
By bijectivity of $\psi$ this forces $M\cdot\tau=\tau$ for all
$\tau\in\PG(1,q)$, hence $M$ is the identity in $\PGL(2,q)$.
Therefore $\varphi$ is injective.

Finally, for $q\geq8$ it is a classical result
(see \cite[Theorem~21.3.17]{Hir1};
also recalled in \cite{CP})
that the full stabilizer of $A_{e}$ in $\PGL(4,q)$ is isomorphic to
$\PGL(2,q)$ and is realized by the above lifted action. Hence
$\varphi(\PGL(2,q))=G_e$, and $G_e\simeq\PGL(2,q)$.
\end{proof}

\medskip
\begin{proof}[Proof of Lemma~\ref{lem:phi-conjugates-singer}]
\textbf{(a)}~
Let $M\in \PGL(2,q)$ be a Singer element. Over $K=\F_{q^2}$,
$M$ has   distinct fixed
points on $\PG(1,K)$, say $\{\gamma,\gamma^q\}$,
and none on $\PG(1,q)$.
Indeed, $M$ cannot fix any point of $\PG(1,q)$.
For if $x\in\PG(1,q)$ were
fixed by $M$, then
$M\in {\rm Stab}_{\PGL(2,q)}(x)$, hence ${\rm ord}(M)$ is a divisor of
$|{\rm Stab}_{\PGL(2,q)}(x)|$.
Since $\PGL(2,q)$ acts transitively on $\PG(1,q)$,
the orbit--stabilizer theorem gives
\[
|{\rm Stab}_{\PGL(2,q)}(x)|=\frac{|\PGL(2,q)|}{|\PG(1,q)|}
=\frac{q(q^2-1)}{q+1}=q(q-1).
\]
But for a Singer element ${\rm ord}(M)=q+1$,
which does not divide $q(q-1)$
(in particular, when $q=2^m$ one has $\gcd(q+1,q(q-1))=1$).
This is a
contradiction, so $M$ has no fixed point on $\PG(1,q)$.
Choose a representative matrix
\[
\widetilde M=\begin{pmatrix}a&b\\ c&d\end{pmatrix}\in\mathrm{GL}(2,q)
\quad\text{for }M,
\]
so that $M$ acts by the fractional linear transformation
\[
t\longmapsto \widetilde M(t)=\frac{at+b}{ct+d}
\quad\text{on }\PG(1,K).
\]
A point $t\in\PG(1,K)$ is fixed by $M$ iff $t=\widetilde M(t)$.
If $t\in K$ (affine coordinate), this is equivalent to
\begin{equation}\label{eq:fixed-quad}
ct^2+(d-a)t-b=0.
\end{equation}
If $c=0$ then $\widetilde M(\infty)=\infty$
\big(since $\infty = [1:0]$, $\widetilde M \left(
                                         \begin{array}{c}
                                           1 \\
                                           0 \\
                                         \end{array}
                                       \right)
=\left(
\begin{array}{c}
                                           a \\
                                           c \\
                                         \end{array}
                                       \right)$\big),
so $\infty$ would be a fixed
point in $\PG(1,q)$, contradicting regularity of $M$.
Hence $c\neq 0$, so
\eqref{eq:fixed-quad} is a genuine quadratic
polynomial over $\F_q$.

Because $M$ has no fixed point in $\PG(1,q)$, \eqref{eq:fixed-quad} has no
root in $\F_q$, hence it is irreducible over $\F_q$. Therefore it has two
distinct roots in $K=\F_{q^2}$, say $\gamma$ and $\gamma'$.
Since the coefficients lie in $\F_q$, Frobenius sends roots to roots, so
$\gamma^q$ is also a root. The two roots are distinct and both lie in
$K\setminus\F_q$, hence $\gamma'=\gamma^q$. This proves (a).

\medskip
\textbf{(b)}
Fix a root $\gamma\in K\setminus\F_q$ of \eqref{eq:fixed-quad}, so the
other root is $\gamma^q$. Define
\[
\phi(t)=\frac{t-\gamma}{t-\gamma^q}\ (t\in K),\qquad \phi(\infty)=1.
\]
This is a M\"{o}bius transformation in $\PGL(2,K)$, hence a bijection of
$\PG(1,K)$. In particular, it is injective on the subset $\PG(1,q)$.
For $t\in\F_q$ we compute using $t^q=t$:
\[
\phi(t)^q=\left(\frac{t-\gamma}{t-\gamma^q}\right)^q
=\frac{t-\gamma^q}{t-\gamma},
\]
so
\[
\phi(t)^{q+1}=\phi(t)\phi(t)^q
=\frac{t-\gamma}{t-\gamma^q}\cdot\frac{t-\gamma^q}{t-\gamma}=1.
\]
Thus $\phi(t)\in U_{q+1}$ for all $t\in\F_q$, and also $\phi(\infty)=1\in U_{q+1}$.
Hence $\phi(\PG(1,q))\subseteq U_{q+1}$.
Since $\phi$ is injective on $\PG(1,q)$ and $|\PG(1,q)|=q+1=|U_{q+1}|$,
it follows that $\phi:\PG(1,q)\to U_{q+1}$ is a bijection, proving (b).

 Identify $\PG(1,K)$ with $K\cup\{\infty\}$ via
$t\leftrightarrow [t:1]$, $0=[0:1]$, and $\infty=[1:0]$.
In order to show that \textbf{(c)} holds, we first give the following
claim:

\smallskip\noindent\textbf{Claim.}
If $g\in\PGL(2,K)$ satisfies $g(0)=0$ and $g(\infty)=\infty$, then there
exists a unique $\beta\in K^\times$ such that
\[
g(u)=\beta u \qquad \text{for all } u\in K\cup\{\infty\},
\]
where by convention $\beta\cdot\infty=\infty$.
Equivalently, $g$ is represented in $\PGL(2,K)$ by the diagonal matrix
$\mathrm{diag}(\beta,1)$.

\smallskip\noindent\emph{Proof of the Claim.}
Choose a representative matrix
$A=\begin{pmatrix}a&b\\ c&d\end{pmatrix}\in\GL(2,q^2)$ for $g$.
Then for $t\in K$ one has
\[
g(t)=\frac{at+b}{ct+d},
\]
and $g(\infty)=\infty$ if and only if $c=0$ (indeed, if $c\neq 0$ then
$g(\infty)=a/c\in K$ is finite). Hence $c=0$.
With $c=0$ we have $g(t)=\frac{at+b}{d}$, where $d\neq 0$ because
$\det(A)=ad\neq 0$.
The condition $g(0)=0$ now gives $0=g(0)=b/d$, hence $b=0$.
Therefore $A$ is diagonal:
$A=\mathrm{diag}(a,d)$ with $a,d\neq 0$, and for all $t\in K$,
\[
g(t)=\frac{a}{d}\,t.
\]
Let $\beta=a/d\in K^\times$. Then $g(u)=\beta u$ for all $u\in K$ and
also $g(\infty)=\infty=\beta\cdot\infty$, as claimed.
Uniqueness of $\beta$ is immediate from $\beta=g(1)$.
\qed

% ----------------------------------------------------------------------
% Insert the following text in the proof of Lemma~\ref{lem:phi-conjugates-singer},
% as the proof of item (c).
% ----------------------------------------------------------------------
\medskip
Proof of \textbf{(c)}.
Put
\[
g \;:=\; \phi\circ M\circ \phi^{-1}\ \in\ \PGL(2,K).
\]
By \textup{(a)}, the two fixed points of $M$ on $\PG(1,K)$ are
$\gamma$ and $\gamma^q$, and by definition of $\phi$ we have
$\phi(\gamma)=0$ and $\phi(\gamma^q)=\infty$.
Hence
$$
g(0)=\phi\bigl(M(\phi^{-1}(0))\bigr)=\phi(M(\gamma))=\phi(\gamma)=0
$$
and
$$
g(\infty)=\phi\bigl(M(\phi^{-1}(\infty))\bigr)
=\phi(M(\gamma^q))=\phi(\gamma^q)=\infty.
$$
By the claim above, there exists a (unique) $\beta\in K^\times$
such that $g(u)=\beta u$ for all $u\in \PG(1,K)$.
Restricting to $\PG(1,q)$ and using \textup{(b)} (namely,
$\phi(\PG(1,q))=U_{q+1}$), we obtain for all $\tau\in\PG(1,q)$:
\[
\phi(M\cdot\tau)\;=\;(\phi\circ M)(\tau)\;=\;(g\circ \phi)(\tau)
\;=\;\beta\,\phi(\tau),
\]
equivalently $g=\phi\circ M\circ\phi^{-1}$ acts
on $U_{q+1}$ as $u\mapsto \beta u$.
Now assume $M$ is represented by a matrix
$\begin{pmatrix}a&b\\ c&d\end{pmatrix}\in\GL(2,q)$.
Next, $\phi$ is represented in $\PGL(2,K)$ by
$\begin{pmatrix}1&-\gamma\\ 1&-\gamma^q\end{pmatrix}$, so
\[
\phi(\infty)=\phi([1:0])=[1:1]=1,
\]
and hence $\phi^{-1}(1)=\infty$.
Using the uniqueness of $\beta$ from
the above claim, we get
\[
\beta \;=\; g(1)\;=\;\phi\bigl(M(\phi^{-1}(1))\bigr)\;=\;\phi(M(\infty))
\;=\;\phi\!\left(\frac{a}{c}\right)
\;=\;\frac{\frac{a}{c}-\gamma}{\frac{a}{c}-\gamma^q}
\;=\;\frac{a-c\gamma}{a-c\gamma^q}.
\]
We see that $\beta\in U_{q+1}$.
Indeed, since $a,c\in\F_q$, we have
$(a-c\gamma)^q=a-c\gamma^q$ and $(a-c\gamma^q)^q=a-c\gamma$, hence
\[
\beta^q
=\left(\frac{a-c\gamma}{a-c\gamma^q}\right)^q
=\frac{a-c\gamma^q}{a-c\gamma}
=\beta^{-1},
\]
so $\beta^{q+1}=1$ and therefore $\beta\in U_{q+1}$.

Finally,   we may represent $g$ by
$\mathrm{diag}(\beta,1)$ in $\PGL(2,K)$,
so the order of $g$ equals the order of
$\beta$ in $K^\times$.
Since $g$ is conjugate to $M$, we have
$\mathrm{ord}(g)=\mathrm{ord}(M)=q+1$,
and thus $\mathrm{ord}(\beta)=q+1$ as required.
\end{proof}

% ============================================================

\end{document}